\documentclass[11pt]{amsart}

\usepackage{amsfonts,amssymb,amsthm,amsmath,amsxtra,amscd,verbatim,eucal}
\usepackage{enumerate}
\usepackage{psfrag}
\usepackage[colorlinks=true,linkcolor=black,citecolor=black,urlcolor=black]{hyperref}

\usepackage[nohug]{diagrams}

\theoremstyle{plain}
\newtheorem{thm}{Theorem}%[section]

\newtheorem*{thm*}{Theorem}
\newtheorem*{prop*}{Proposition}

\newtheorem{lem}[thm]{Lemma}
\newtheorem{prop}[thm]{Proposition}
\newtheorem{cor}[thm]{Corollary}

\theoremstyle{definition}

\newtheorem{rmk}[thm]{Remark}
\newtheorem{eg}[thm]{Example}

\newcommand{\excise}[1]{}

\renewcommand{\setminus}{\smallsetminus}

\def\ZZ{{\mathbb Z}}
\def\NN{{\mathbb N}}

\def\CC{{\mathbb C}}

\def\R{{\mathbb R}}

\def\PP{{\mathbb P}}

\newcommand{\dsubseteq}{\ensuremath{\,\subseteq\,}}

\newcommand{\dleq}{\ensuremath{\,\leq\,}}
\newcommand{\dgeq}{\ensuremath{\,\geq\,}}
\newcommand{\equ}{\ensuremath{ \,=\, }}
\newcommand{\st}[1]{\ensuremath{ \left\{ #1 \right\} }}

\def\OO{\mathcal{O}}
\newcommand{\vol}[2]{\ensuremath{{\rm vol}_{#1}\left( #2 \right) } }

\DeclareMathOperator{\Pic} {Pic}
\DeclareMathOperator{\ord} {ord}

\begin{document}
%%%%%%%%%%%%%%%%%%%%%%%%

%%%%%%%%%%%%%%%%%%%%%%%%%%%%%%%%%%%%%%%%%%%%%%%%%%%%%%%%%%
\title{Okounkov bodies of finitely generated divisors}

\thanks{During this project AK was partially supported by  DFG-Forschergruppe 790 ``Classification of Algebraic Surfaces and Compact Complex Manifolds'', and the OTKA Grants 77476 and  81203 by the Hungarian Academy of Sciences.  DA was partially supported by NSF Grant DMS-0902967.}

\author{Dave Anderson}
\address{Department of Mathematics\\University of Washington\\Seattle, WA 98195}
\email{dandersn@math.washington.edu}
\author{Alex K\"uronya}
\address{Budapest University of Technology and Economics, Department of Algebra, P.O. Box 91, H-1521 Budapest, Hungary}
\address{Albert-Ludwigs-Universit\"at Freiburg, Mathematisches Institut, Eckerstra\ss e 1, D-79104 Frei\-burg, Germany}
\email{{\tt alex.kueronya@math.uni-freiburg.de}}
\author{Victor Lozovanu}
\address{Queen's University, Department of Mathematics and Statistics,  Kingston, Ontario, K7L3N6, Canada}
\email{vicloz@mast.queensu.ca}

\date{September 11, 2012}

\begin{abstract}
We show that the Okounkov body of a big divisor with finitely generated section ring is a rational simplex, for an appropriate choice of flag; furthermore, when the ambient variety is a surface, the same holds for every big divisor.  Under somewhat more restrictive hypotheses, we also show that the corresponding semigroup is finitely generated.
\end{abstract}

\maketitle
%%%%%%%%%%%%%%%%%%%%%%%%%%%%%%%%%%%%%%%%%%%%%%%%%%%%%%%%%%

Consider an $n$-dimensional irreducible projective variety $X$, with a big divisor $L$.  (All divisors are tacitly assumed to be Cartier, and we always work over the complex numbers.)  
The \emph{Newton-Okounkov body}---or \emph{Okounkov body}, for short---of $L$ is a convex body $\Delta_{Y_\bullet}(L) \subseteq \R^n$ 
which captures interesting geometric information about $L$.  For example, the volume of $L$ (defined as the rate of growth of sections of $mL$) is equal to the Euclidean volume of $\Delta_{Y_\bullet}(L)$, up to a normalizing factor of $n!$.  The construction of $\Delta_{Y_\bullet}(L)$ depends on the choice of an \emph{admissible flag} of subvarieties of $X$, 
that is, a chain of irreducible subvarieties $X=Y_0 \supset Y_1 \supset \cdots \supset Y_n$, such that each $Y_i$ has codimension $i$ in $X$ 
and is nonsingular at the point $Y_n$.  We refer to the seminal papers of Kaveh-Khovanskii \cite{KK08} and Lazarsfeld-Musta\c t\u a \cite{LM08} 
for more details, but emphasize that in general the Okounkov body depends quite sensitively on the choice of flag.  

In many cases of interest, including the original context considered by Okounkov \cite{O96}, the convex body $\Delta_{Y_\bullet}(L)$ is in fact 
a rational polytope.  However, this is not always true: even when $X$ is a  Mori dream space and $L$ is an ample divisor, there are examples where one can choose a flag to produce a non-polyhedral Okounkov body (see \cite{KLM1}).  

The question thus arises: given a big divisor $L$ on an irreducible projective variety $X$, is there an admissible flag with respect to which the corresponding Okounkov body is a rational polytope? 

For the answer to be ``yes'', a necessary condition is that the volume of $L$ be a rational number.  This is guaranteed in the case when the section ring of $L$ is finitely generated, so let us impose this hypothesis as well.

Our main result is an affirmative answer to this question for divisors with finitely generated section rings. 

\begin{thm}\label{thm:main_intro}
Let $X$ be a normal complex projective variety, and let $L$ be a big divisor with finitely generated section ring.  Then there exists an admissible flag $Y_\bullet$ on $X$ such that the Newton-Okounkov body $\Delta_{Y_\bullet}(L)$ is a rational simplex. 
\end{thm}

\noindent
In fact, we do not require normality if $L$ is globally generated (Theorem~\ref{thm:big and globally generated}).  The basic idea is to construct the flag by intersecting general elements of the linear series $|L|$.  This works directly when $L$ is ample, and can be modified for the other cases.

As A.~Khovanskii observed upon hearing our results, Theorem~\ref{thm:main_intro} shares a philosophical connection with the fact that after a change of variables, any given polynomial has a simplex as its Newton polyhedron.

In the case of surfaces, any big divisor has a unique Zariski decomposition (see \cite[Chapter 2.3.E]{PAG}).  The volume of a big divisor on a smooth surface is therefore always a rational number (\cite[Corollary 2.3.22]{PAG}).  Using the existence of Zariski decomposition, we show in Corollary~\ref{cor:surface} that for any big divisor on a smooth surface there exists a flag such that the Okounkov body is a rational simplex.

Before delving into the proofs, we briefly review the construction of the Newton-Okounkov body.  As always, $X$ is an $n$-dimensional irreducible projective variety, and $L$ is a big divisor on $X$.  Orders of vanishing along components of the flag give rise to a rank $n$ valuation $\nu=\nu_{Y_\bullet}$, which takes a nonzero section $s\in H^0(X,mL)$ to an integer vector $\nu(s) = (\nu_1(s),\ldots,\nu_n(s)) \in \ZZ^n$.  Collecting these for all multiples $m$, one obtains a semigroup
\[
  \Gamma_{Y_\bullet}(L) = \{ (m,\nu(s)) \,|\, s \in H^0(X,mL)\setminus\{0\} \} \subseteq \NN \times \ZZ^n.
\]
Finally, the convex body is obtained by slicing the closed convex hull of the semigroup:
\[
  \Delta_{Y_\bullet}(L) = \overline{\mathrm{Conv}(\Gamma_{Y_\bullet}(L))} \cap (\{1\}\times\R^n).
\]

From the construction, it is clear that the Okounkov body is a rational polytope whenever the semigroup $\Gamma_{Y_\bullet}(L)$ is finitely generated.  In general, finite generation of the semigroup is a quite subtle issue, with interesting consequences: for instance, it implies the existence of a completely integrable system on $X$, by recent work of Harada and Kaveh \cite{HK12}.  Our methods lead to a criterion for $\Gamma_{Y_\bullet}(L)$ to be finitely generated, under somewhat more restrictive hypotheses (Proposition~\ref{prop:fg}).

\bigskip

We now turn to the proof of Theorem~\ref{thm:main_intro}.  We will give the argument in stages: first for ample divisors, then for big and globally generated divisors, and finally the more general case of the theorem.

The key ingredient is a fact relating slices of the Okounkov body to restrictions along the flag $Y_\bullet$.  For this, we need some notation: given an admissible flag $X = Y_0 \supset Y_1 \supset \cdots \supset Y_n$, the \emph{restricted flag} $Y_\bullet|_{Y_1}$ is just the flag $Y_1 \supset \cdots \supset Y_n$, considered as an admissible flag on $Y_1$.  The \emph{augmented base locus} of a divisor $D$, denoted $\mathbf{B}_+(D)$, is by definition the stable base locus of a perturbation $D-\epsilon A$, for an ample divisor $A$ and a small $\epsilon>0$; see \cite[Chapter 10.3]{PAG} for details and examples.  (The reader unfamiliar with these notions need not worry, since in our main applications $\mathbf{B}_+(D)$ will be empty.)

\begin{prop}[{\cite[Proposition 3.1]{KLM1}}]\label{slices}
Let $X$ be an irreducible projective variety of dimension $n$, equipped with an admissible flag $Y_{\bullet}$.  Suppose that $D$ is a divisor such that $Y_1\nsubseteq \mathbf{B}_+(D)$.  Then
\[
 \Delta_{Y_{\bullet}}(D)\cap \big(\{s\}\times \mathbb{R}^{n-1}\big) = \Delta_{Y_{\bullet}|_{Y_1}}((D-sY_{1})|_{Y_{1}})
\]  
whenever $D-sY_{1}$ is ample.
\end{prop}

\begin{rmk}
The original statement found in \cite[Proposition 3.1]{KLM1} does not include the additional condition about the augmented base locus, but without this hypothesis, there is a gap in the proof: it relies on a statement about the slices of Okounkov bodies (\cite[Theorem 4.26]{LM08}) whose proof requires this additional condition.  In all the applications of this slicing techinique considered in \cite{KLM1}, these additional condition is satisfied automatically. The proof of Proposition~\ref{slices}, as stated here, is the same as the proof given in \cite{KLM1}, which also goes through for $X$ irreducible but not necessarily smooth.
\end{rmk}

Our starting point is the following useful observation. 

\begin{prop}\label{prop:very ample}
 Let $L$ be a very ample divisor on $X$ and let $E_1,\dots,E_n$ be sections in the linear series $|L|$ defining an admissible flag $Y_\bullet$, by $Y_k = E_1\cap\dots\cap E_k$ for $1\leq k\leq n-1$, and $Y_n$ a point in $E_1\cap \cdots \cap E_n$.  Then $\Delta_{Y_\bullet}(L)$ is a simplex defined by the inequalities
\[
 x_1\dgeq 0\, ,\, \dots \,,\, x_n\dgeq 0\, ,\, x_1+\dots+x_{n-1}+\frac{1}{c}x_n \dleq 1,
\]
or equivalently, it is the convex hull of the $n+1$ points
\[
  (0,0,\ldots,0), \, (1,0,\ldots,0), \, (0,1,\ldots,0), \ldots, \, (0,\ldots,0,c),
\]
where $c=(L^n)$ is the degree of $L$.
\end{prop}

By Bertini's theorem, a general choice of sections $E_i\in |L|$ satisfies the hypothesis of Proposition~\ref{prop:very ample},
 so this yields a special case of Theorem~\ref{thm:main_intro}.

\begin{proof}
We proceed by induction on dimension of $X$. %By taking  the elements $E_i\in |L|$ to be general, the flag  $Y_\bullet$  defined above will be admissible.
Because $Y_1\sim L$, it follows from the construction of the Okounkov body that $\Delta_{Y_{\bullet}}(L)\subseteq [0,1]\times\mathbb{R}^{n-1}$.  Since $L-tY_1$ is ample for all $0\leq t<1$, and $\mathbf{B}_+(L)=\varnothing$, we have
\[
 \Delta_{Y_\bullet}(L)|_{x_1=t} \equ \Delta_{{Y_\bullet}|_{Y_1}}((L-tY_1)|_{Y_1}) = (1-t)\cdot \Delta_{{Y_\bullet}|_{Y_1}}(L|_{Y_1}),
\]
for $0\leq t<1$; the first equality follows from Proposition~\ref{slices}, and the second from the homogeneity property of Okounkov bodies \cite[Proposition~4.1, Remark~4.15]{LM08}. Since $\Delta_{Y_\bullet}(L)$ is a closed convex body, we obtain
\[
 \Delta_{Y_\bullet}(L)|_{x_1=t} \equ (1-t)\cdot \Delta_{{Y_\bullet}|_{Y_1}}(L|_{Y_1})
\]
for all $0\leq t\leq 1$. 

To invoke the induction hypothesis, we need to verify that $L|_{Y_1}$ and the flag $Y_\bullet|_{Y_1}$ satisfy the hypotheses of the proposition.  This is easily done, however, since $L|_{Y_1}$ remains very ample, 
\[
 E_2|_{E_1}\, ,\, E_3|_{E_1}\, ,\dots ,\, E_n|_{E_1}
\]
are sections of $L|_{Y_1}$, and $Y_\bullet|_{Y_1}$ remains admissible. 

We have reduced to the case of a curve.  Assume now that $\dim X=1$, $L$ is a very ample line bundle on $X$.  Let $Y_\bullet$ be the flag $X\supset E_n$, where $E_n$ is a point
on $X$, and $L$ is numerically equivalent to $c\cdot E_n$, where $c=\deg L$ is a positive integer.  By \cite[Example~1.13]{LM08}, 
\[
 \Delta_{Y_\bullet}(L) \equ [0,c] ,
\]
which finishes the proof. 
\end{proof}

Since a sufficiently large multiple of an ample divisor is very ample, Proposition~\ref{prop:very ample}, 
together with the homogeneity of Okounkov bodies, implies the case of Theorem~\ref{thm:main_intro} where $L$ is ample:

\begin{cor}
Let $X$ be an irreducible projective variety, $L$ an ample divisor on $X$.  Then there exists an admissible flag on $X$ with respect to which the Okounkov body of $L$ 
is a rational simplex. 
\end{cor}

\begin{rmk}
In the ample case, there is an alternative argument that does not rely on Proposition~\ref{slices}.  While finishing this article, we learned that H.~Sepp\"anen has recently also found this more direct proof, which appears in \cite{S12}\footnote{The result does not appear in the first two versions of the cited preprint, and was first posted on June 6, 2012.}.  The argument we give here allows approximation of Okounkov bodies (see, e.g., the proof of Proposition~\ref{prop:restriction for big and globally generated}), a technique which may have wider applications.
\end{rmk}

Our next aim is to prove a generalization to the globally generated case. 

\begin{thm}\label{thm:big and globally generated}
Let $X$ be an irreducible projective variety, with a big and globally generated divisor $L$.  There exists an admissible flag $Y_\bullet$ on $X$ with respect to which $\Delta_{Y_\bullet}(L)$ is a rational simplex. 
\end{thm}

The main idea of the proof is to choose an arbitrary but sufficiently positive divisor $A$ on $X$, and approximate $\Delta_{Y_\bullet}(L)$ by the Okounkov bodies 
$\Delta_{Y_\bullet}(\frac{1}{m}A+L)$ as $m$ tends to infinity.  To accomplish this, we will need some technical preliminaries.

\begin{lem}\label{lem:nested}
Let $L$ and $D$ be divisors on an irreducible projective variety, with $L$ big and $D$ semiample.  Let $Y_\bullet$ an admissible flag.  Then, for integers $m_2\geq m_1 >0$, we have
\begin{align}
 \Delta_{Y_\bullet}(\frac{1}{m_2}D+L) &\dsubseteq \Delta_{Y_\bullet}(\frac{1}{m_1}D+L), \label{eq1} \\ 
\intertext{and}
 \bigcap_{m=1}^{\infty}\Delta_{Y_\bullet}(\frac{1}{m}D+L) &\equ \Delta_{Y_\bullet}(L) .\label{eq2}
\end{align}
\end{lem}

\begin{proof}
Making use of the homogeneity property for Okounkov bodies, we can replace $D$ and $L$ by suitable multiples, and thereby assume that $D$ is base point free. In this case, the existence of global Okounkov cone  implies the inclusion
\[
\Delta_{Y_{\bullet}}(\frac{1}{m_2}D+L)+\Delta_{Y_{\bullet}}((\frac{1}{m_1}-\frac{1}{m_2})D)\ \subseteq \ \Delta_{Y_{\bullet}}(\frac{1}{m_1}D+L).
\]
(See Theorem~4.5 and the proof of Corollary~4.12 in \cite{LM08}.)  
Since $D$ is base point free, the body $\Delta_{Y_{\bullet}}((\frac{1}{m_1}-\frac{1}{m_2})D)$ contains the origin and thus the inclusion in \eqref{eq1} follows.  The equality in \eqref{eq2} is an immediate consequence of the inclusion in \eqref{eq1}, together with the continuity property arising from the existence of the global Okounkov cones.
\end{proof}

\begin{prop}\label{prop:restriction for big and globally generated}
Let $L$ be a big and globally generated divisor on $X$, $Y_\bullet$ an admissible flag with the property that $L\sim Y_1$. Then
\[
\Delta_{Y_\bullet}(L)|_{x_1=t} \equ  (1-t)\Delta_{Y_\bullet|_{Y_1}}(L|_{Y_1}) 
\]
for all $0\leq t\leq 1$. 
\end{prop}
\begin{proof}
It is enough to prove the equality for $t\in [0,1)$, since the case $t=1$ follows from the fact that $\Delta_{Y_{\bullet}}(L)$ is a closed convex set. We first claim that if $A$ is an ample divisor on $X$, then
\[
 \Delta_{Y_\bullet|_{Y_1}}((\frac{1}{m}A+L-tY_1)|_{Y_1}) \equ \Delta_{Y_\bullet}(\frac{1}{m}A+L)|_{x_1=t}
\]
for all positive integers $m$ and all $t\in [0,1)$. Granting this, the proposition follows from the equality \eqref{eq2} of Lemma~\ref{lem:nested} and the homogeneity property of Okounkov bodies.

To prove the claim, after scaling we arrive at the equivalent statement
\[
 \Delta_{Y_\bullet|_{Y_1}}((A+m(L-tY_1))|_{Y_1}) \equ \Delta_{Y_\bullet}(A+mL)|_{x_1=t}\ .
\]
For any $m\in\mathbf{N}$ define $D_m:=A+mL$. Then for any $m\in\mathbf{N}$, $D_m-tY_1$ is ample for all $t\in[0,1)$, and $\mathbf{B}_+(D_m)=\varnothing$, since $D_m$ is also ample.  The displayed equality now follows directly from Proposition~\ref{slices} applied to the divisor $D_m$.
\end{proof}

\begin{proof}[Proof of Theorem~\ref{thm:big and globally generated}]
Let $E_1,\ldots ,E_{n-1}\in |L|$ be general elements of the linear series. Using induction and applying the Bertini theorem for big and globally divisors at each step (see \cite[Theorem 3.3.1]{PAG}), the flag 
\[
 Y_1 \equ E_1\, ,\, \dots \, ,\, Y_{n-1} \equ E_1\cap\dots\cap E_{n-1}\, ,\, Y_n \equ \st{x}
\]
is admissible, where $x\in X$ is chosen to be a smooth point of the curve $Y_{n-1}$.  As  before, by using Proposition~\ref{prop:restriction for big and globally generated} we reduce the question to dimension one less. We observe that the conditions of the Theorem are fulfilled for the restrictions to $Y_1$, hence induction on $\dim X$ takes us back to the case of curves. On a curve, however, the Okounkov body of a complete linear series is a line segment with rational endpoints. 
\end{proof}

Now we can give the proof of the main theorem, where $L$ is a big divisor with finitely generated section ring.  For this, we require $X$ to be normal.

\begin{proof}[Proof of Theorem~\ref{thm:main_intro}]
The main idea is to approximate some power of $L$ with a globally generated divisor on a modification of $X$.  Let $R(X,L) = \bigoplus_{m\geq 0} H^0(X,mL)$ be the section ring of $L$.  Following \cite[Example 2.1.31]{PAG}, if $X$ is normal and $R(X,L)$ is generated in degree $p>0$, one can construct a proper birational morphism $f\colon X'\to X$ from a normal projective variety, and an effective divisor $N$ on $X'$ with the following properties:
\renewcommand{\theenumi}{\alph{enumi}}
\begin{enumerate}
\item the divisor $D:=f^*(pL)-N$ is big and globally generated, and \label{p1}

\smallskip

\item for all $m\geq 0$, $H^0(X', \OO_{X'}(mD))  =  H^0(X,\OO_X(mpL))$, the identification being given by $f_*$. \label{p2}
\end{enumerate}
\renewcommand{\theenumi}{\arabic{enumi}}
This tells us that outside the support of $N$ and the exceptional locus of $f$, the zero-locus of a global section from $|mD|$ is isomorphic to the zero-locus of a global section from $|mpL|$, and vice versa. 

With this in hand, construct an admissible flag $Z_{\bullet}$ on $X'$ as in the proof of Theorem~\ref{thm:big and globally generated}, using general sections of $D$ with the additional condition that $Z_n \notin\textup{Supp}(N)\cup \textup{Exc}(f)$. Take $Y_\bullet$ be the 
image of $Z_\bullet$ under $f$.  By property \eqref{p2} above, the Okounkov bodies $\Delta_{Z_{\bullet}}(D)$ and $\Delta_{Y_{\bullet}}(L)$ are equal.  Since $\Delta_{Z_\bullet}(D)$ is a rational simplex by Theorem~\ref{thm:big and globally generated}, we are done.
\end{proof}
\begin{rmk}
An interesting consequence of \cite[Example 2.1.31]{PAG} is that if $D$ is a Cartier divisor on an $n$-dimensional normal projective variety whose section ring is generated in degree $d$, then the denominator of $\vol{X}{D}$ divides $d^n$.  We have not seen this stated in the literature.

In particular, if $D$ is a divisor with $\vol{X}{D}=1/p$ where $p$ is a prime number, 
then $R(X,D)$ cannot be generated in degree less than $p$.  
\end{rmk}

\medskip

When $X$ is a surface, we can say more.

\begin{prop}\label{prop:surface}
Let $X$ be a smooth projective surface and $L$ be a big and nef line bundle on $X$.  Then there exists an irreducible curve $Y_1\subseteq X$ such that the Okounkov body $\Delta_{Y_\bullet}(L)$ is a rational simplex, where $Y_2$ is a general point on the curve $Y_1$.
\end{prop}

\begin{proof}
If $L$ is semi-ample, this follows from Theorem~\ref{thm:big and globally generated}.  So, we can assume that $L$ is not semi-ample, i.e. the stable base locus $\textbf{B}(L)\neq 0$. Let $C_1,\ldots ,C_r$ be the only irreducible curves contained in $\textbf{B}(L)$.  Then there exists positive integers $m,a_1,\ldots ,a_r\in\mathbf{Z}$ such that the base locus of the divisor $mL-\sum_{i=1}^{i=r}a_iC_i$ is a finite set.  By Zariski-Fujita theorem \cite[Remark 2.1.32]{PAG} this divisor is semi-ample.  Consequently, taking large enough multiples, we can assume that $mL-\sum_{i=1}^{i=r}a_iC_i$ is a base point-free divisor and $\textbf{B}(L)=C_1\cup\ldots \cup C_r$.  Taking this into account we choose $Y_1\in |mL-\sum_{i=1}^{i=r}a_iC_i|$ to be an irreducible curve.

Furthermore we know that $(L.C_i)=0$ for all $i=1,\ldots ,r$. This follows from the fact that $\mathbf{B}(L)\subseteq\mathbf{B}_+(L)=\textup{Null}(L)$, where the inclusion follows from \cite[Definition 10.3.2]{PAG} and the equality from \cite[Thereom 10.3.5]{PAG}. Since $(L^2)>0$, as $L$ is big and nef, then we also know that the matrix $||(C_i.C_j)||_{i,j}$ is negative definite and the divisor $\sum_{i=1}^{i=r}b_iC_i$, for any $b_i\in\mathbb{R}_+$, is pseudo-effective but not big. 

For the proof we consider the family of divisors $D_t = L-tY_1$ for all $t\geq 0$. The claim is that $D_t$ is effective iff $0\leq t\leq 1/m$. For this notice that 
\[
 D_t \equ L-tY_1 \sim L-t(mL-\sum_{i=1}^{r}a_iC_i) \equ (1-mt)L + t\cdot\sum_{i=1}^{r}a_iC_i\ .
\]
By this description $D_t=L-tY_1$ is effective for $0\leq t\leq 1/m$. If $t>1/m$, set $\epsilon=t-1/m$. Since $(L\cdot C_i)=0$ for all $i=1,\ldots , r$, and $(L^2)>0$, then the intersection number
\[
(D_t\cdot L) \equ ( (-\epsilon L+ (1+\epsilon)\sum_{i=1}^{r}a_iC_i)\cdot L) \equ -\epsilon(L^2) < 0\ .
\]
 Taking into account that $L$ is nef, then this implies that $D_t$ is not effective when $t>1/m$.

 Above, we have seen that $D_t\sim (1-mt)L + t\cdot\sum_{i=1}^{r}a_iC_i$. Thus the divisors $P_t \equ (1-mt)L$ and $N_t \equ t\sum_{i=1}^{r}a_iC_i$ form the Zariski decomposition of $D_t$ for all $0\leq t\leq 1/m$. This follows since $D_t = P_t+N_t$, $(P_t\cdot N_t)=0$, $P_t$ is nef (since $L$ is), $N_t$ is a negative definite cycle, and the uniqueness of Zariski decompositions. This and the claim above say that the divisor  $D_t=L-tY_1$ stays in the same Zariski chamber for all $0\leq t \leq 1/m$. 

Now, going back to the Okounkov body, we know, by \cite[Theorem 6.4]{LM08}, that if we choose a smooth point $x\in Y_1$, then we have the following description
\[
\Delta_{(Y_1 ,x)}(L) \ = \ \{ (t,y)\in\R^2 \ | \ a\leq t\leq \mu_{Y_1}(L), \textup{ and } \alpha (t)\leq t\leq \beta (t) \ \}
\]
The choice of $Y_1$ forces $a=0$ and the claim implies that $\mu_{Y_1}(L)=1/m$. The proof of \cite[Theorem 6.4]{LM08} says that $\alpha(t)=\textup{ord}_x(N_t|_{Y_1})$ and $\beta(t)=\textup{ord}_x(N_t|_{Y_1})+(Y_1 .P_t)$. Thus by choosing $x\notin \textbf{B}(L)$, which is always possible since the curve $Y_1$ moves in base point free linear series, we obtain $\alpha(t)=0$ and $\beta(t)=(1-mt)(Y_1 .L)$. Thus $\Delta_{Y_\bullet}(L)$ is a simplex. 
\end{proof}

\begin{rmk}
To our knowledge, this is the first case where the rational polyhedrality of the Okounkov body of  a big but non-finitely generated divisor is established.  The above proof is very surface-specific; it is an open question whether big and nef (but non-finitely generated) divisors can have rational polytopes as Okounkov bodies if the dimension of the underlying variety is at least three.
\end{rmk}

\begin{cor}\label{cor:surface}
Let $X$ be a smooth projective surface, $D$ a big divisor on $X$. Then there exists an admissible flag $Y_\bullet$ with respect to which
$\Delta_{Y_\bullet}(L)$ is a rational simplex. 
\end{cor}
\begin{proof}
Let $D=P+N$ be the Zariski decomposition of $D$. By \cite[Corollary 2.2]{LS11}, $\Delta_{Y_\bullet}(D)$ is a rational translate of $\Delta_{Y_\bullet}(P)$. By Proposition \ref{prop:surface} there exists a flag $Y_\bullet$ for which $\Delta_{Y_\bullet}(P)$ is a rational simplex, but then so is $\Delta_{Y_\bullet}(D)$ for the same
flag. 
\end{proof}

\bigskip

We conclude with some remarks about the semigroup $\Gamma_{Y_\bullet}(L)$.  For simplicity, we will assume $X$ is nonsingular, and we restrict attention to the situation of Proposition~\ref{prop:very ample}, so $L$ is very ample.  First, consider the projection of $\Gamma_{Y_\bullet}(L)$ onto $\NN \times \ZZ^{n-1}$:
\[
  \Gamma' = \{ (m,\nu_1(s),\ldots,\nu_{n-1}(s)) \,|\, s\in H^0(X,mL)\setminus\{0\} \}.
\]
The proof of Proposition~\ref{prop:very ample} shows that $\Gamma'$ is finitely generated; in fact, it is simply the span of the standard basis vectors.  It is therefore natural to ask whether the flag $Y_\bullet$ can be chosen so that $\Gamma_{Y_\bullet}(L)$ itself is finitely generated.  A general answer to this question would have interesting ramifications: when $\Gamma_{Y_\bullet}(L)$ is finitely generated, $X$ admits a flat degeneration to the corresponding toric variety \cite{A10}, which in turn leads to an integrable system on $X$ \cite{HK12}.

\begin{prop}\label{prop:fg}
Let $X$ be a nonsingular projective variety of dimension $n$, and let $L$ be a very ample divisor.  
Suppose that, under the embedding $X \hookrightarrow \PP = \PP(H^0(X,L))$, there exist  linear subspaces $W_n \subseteq W_{n-1} \subseteq \PP$, of codimensions $n$ and $n-1$, respectively, such that the set-theoretic intersection $Y_n = X\cap W_n$ is a single point and the scheme-theoretic intersection $W_{n-1} \cap X$ is reduced and irreducible.  Then there is a flag $Y_\bullet$, with $Y_k$ an irreducible Cartier divisor in $Y_{k-1}$ for $1\leq k\leq n-1$, such that the semigroup $\Gamma_{Y_\bullet}(L)$ is finitely generated.
\end{prop}

We need a lemma:

\begin{lem}\label{lem:cm}
Let $X \subset \PP^N$ be an $n$-dimensional irreducible Cohen-Macaulay variety, nonsingular except possibly at a point $p$.  Assume there is a codimension $n$ linear space $W_n \subseteq \PP^N$ such that $W_n \cap X = \{p\}$.  Then there is a hyperplane $W_1 \subseteq\PP^N$ containing $W_n$, such that $Y_1 = W_1 \cap X$ is nonsingular except possibly at $p$.  Moreover, $Y_1$ is Cohen-Macaulay of pure dimension $n-1$, hence normal and irreducible when $n\geq 3$.
\end{lem}

\begin{proof}
This is a standard Kleiman-Bertini argument, set up as follows.  When $n=1$, there is nothing to prove, so assume $n\geq 2$.  Let $U = \PP^N \setminus W_n$, and let $G \subseteq GL_{N+1}$ be the parabolic subgroup stabilizing $W_n$.  Then $G$ is connected, and acts transitively on $U$.  Choose any hyperplane $W_1 \supset W_n$, and let $W_1^\circ = W_1 \setminus W_n$.  Consider the diagram
\begin{diagram}
   &   &  \Gamma & \rTo & X \setminus \{p\} \\
   & \ldTo^p &  \dInto &     &  \dInto  \\
 G & \lTo & G \times W_1^\circ & \rTo^a & U,
\end{diagram}
where the square is cartesian.  Since $G$ acts transitively, the action map $a$ is smooth, and it follows that $\Gamma$ is also smooth.  By generic smoothness, a general fiber of $p$ is also smooth, but such a fiber is $(g\cdot W_1^\circ )\cap (X\setminus \{p\})$ by construction.  Now replace $W_1$ with such a translate $g\cdot W_1$; then $Y_1 = X \cap W_1$ is smooth away from $p$, and has pure codimension $1$ in $X$.  It follows that $Y_1$ is Cohen-Macaulay.  When $n\geq 3$, $Y_1$ is also nonsingular in codimension $1$, so it is normal.  This implies irreducibility of $Y_1$, e.g., by \cite{H62}.
\end{proof}

\begin{proof}[Proof of Proposition~\ref{prop:fg}]
Fix $W_n$ as in the hypothesis.  Applying Lemma~\ref{lem:cm} inductively, we obtain linear subspaces $\PP \supset W_1 \supset \cdots \supset W_n$ such that $Y_k = X \cap W_k$ is Cohen-Macaulay, nonsingular away from $Y_n$, irreducible for $k\leq n-2$ and Cartier in $Y_{k-1}$ for $k\leq n-1$.  The assumption that some $W_{n-1}$ intersect $X$ in a reduced and irreducible variety $Y_{n-1}$ implies that the same is true for a general linear subspace, so in fact we can take $Y_{n-1}$ to be irreducible as well.  Since $Y_{n-1}$ may be singular at $Y_n$, the flag $Y_\bullet$ is not admissible in the sense of \cite{LM08}, but it is sufficient to define a valuation in the sense of \cite{KK08}.  Specifically, given a nonzero section $s_k \in H^0(Y_k,D)$, the usual order function $\ord_{Y_{k+1}}(s_k)$ defines $\nu_{k+1}$ for $0 \leq k \leq n-2$; for the last step, define $\nu_n$ to be the local multiplicity $\dim_\CC( \OO_{Y_{n-1},Y_n}/(s_{n-1}) )$.

Choose nonzero sections $w_0, w_1,\ldots,w_n$ in $H^0(X,L)$ such that $w_0$ is not identically zero on $W_n$, and $W_k = \{ w_1 = \cdots = w_k = 0\}$ for $1\leq k\leq n$.  Then
\[
  \nu(w_0) = (0,0,\ldots,0,0), \; \nu(w_1) = (1,0,\ldots,0,0), \; \ldots, \; \nu(w_{n-1}) = (0,0,\ldots,1,0),
\]
and $\nu(w_n) = (0,0,\ldots,0,c)$, where $c = \deg L$.  Since these are the vertices of $\Delta_{Y_\bullet}(L)$ (by Proposition~\ref{prop:very ample}), finite generation of the semigroup $\Gamma_{Y_\bullet}(L)$ follows from \cite[Proposition~4.1]{A10}.
\end{proof}

To see why the condition on $W_{n-1}$ is necessary, consider the following simple example.

\begin{eg}
Let $X=\PP^2$, with $L=\OO(2)$ giving the Veronese embedding in $\PP(H^0(X,L))=\PP^5$.  Take coordinates $u,v,w$ on $\PP^2$ and $x_0,\ldots,x_5$ on $\PP^5$, so the embedding is
\[
  [u,v,w] \mapsto [u^2, uv, v^2, vw, w^2, uw].
\]
The linear space $W_2 = \{ x_2-x_5=x_4=0 \}$ meets $X$ in the single point $[1,0,0]$, and $W_1 = \{ x_2-x_5\}$ meets $X$ in a smooth conic, so this choice does produce a flag as desired.  (The corresponding Okounkov body is the triangle with vertices at $(0,0)$, $(1,0)$, and $(0,4)$.)  However, the linear space $W'_2=\{x_0=x_2=0\}$ meets $X$ in the single point $[0,0,1]$, but every choice of $W_1 = \{ ax_0 + bx_2 = 0\}$ meets $X$ in a degenerate conic $\{au^2 + bv^2 = 0\}$.
\end{eg}

Let $\Delta_c$ be the simplex appearing in Proposition~\ref{prop:very ample}, so $c=\vol{X}{L} = (L^n)$.  The corresponding (normal) toric variety $X(\Delta_c)$ is isomorphic to the weighted projective space $\PP(1,\ldots,1,c)$.  Combining Proposition~\ref{prop:fg} with \cite[Theorem~1.1]{A10}, we obtain the following:

\begin{cor}\label{cor:toric degeneration}
Let $X$ be nonsingular and let $L$ be a very ample line bundle.  Assume the hypotheses of Proposition~\ref{prop:fg}.  
Then $X$ admits a flat degeneration to a toric variety whose normalization is the weighted projective space $\PP(1,\ldots,1,c)$, where $c=\deg(L)$.
\end{cor}

The hypothesis in Proposition~\ref{prop:fg} is quite restrictive.  Even for curves of genus at least two, it fails for a very general set of very ample divisors.  (However, for each degree $d>0$, there is a dense subset of $L$ in $\Pic^d(X)$ for which the hypothesis of Proposition~\ref{prop:fg} does hold; this is an easy and amusing exercise in relating the problem to torsion points on the Jacobian.)  We are led to ask the following question: For which (nonsingular) projective varieties $X$ does there exist \emph{some} very ample divisor $L$, and some codimension $n$ linear space $W_n \subset \PP(H^0(X,L))$, such that $X \cap W_n$ is a single point?

\bigskip
\noindent
{\it Acknowledgements.}  We are grateful to Stefan Kebekus for helpful discussions, and to Megumi Harada, Kiumars Kaveh, and Askold Khovanskii for organizing the Oberwolfach mini-workshop ``Recent developments in Newton-Okounkov bodies,'' where this collaboration began.

%%%%%%%%%%%%%%%%%%%%%%%%

%%%%%%%%%%%%%%%%%%%%%%%%

%%%%%%%%%%%%%%%%%%%%%%%%
\end{document}